\documentclass[preprint,12pt,authoryear]{elsarticle}

\usepackage{amssymb}
\usepackage{amsmath}
\usepackage{indentfirst}
\usepackage{amsfonts}
\usepackage{mathtools}
\usepackage{float}
\usepackage{graphicx}
\usepackage[graphicx]{realboxes}
\graphicspath{{figure/}}
\usepackage{amsthm}
\usepackage{color}
\usepackage[vlined,ruled,linesnumbered]{algorithm2e}
\usepackage{algpseudocode,float}
\usepackage{tikz}
\usetikzlibrary{positioning, shapes.geometric, calc}
\usepackage{makecell}
\usepackage{siunitx}
\usepackage{multirow}
\usepackage[reftex]{theoremref}
\usepackage{bm}
\usepackage{sepfootnotes}
\usepackage[english]{babel}
\usepackage{hyperref}

\journal{}

\newtheorem{theorem}{Theorem}[section]
\newtheorem{lemma}[theorem]{Lemma}

\newtheorem{proposition}[theorem]{Proposition}
\newtheorem{definition}[theorem]{Definition}

\newtheorem{example}[theorem]{Example}
\newtheorem{remark}[theorem]{Remark}

\newcommand{\rnum}{\mathbb{R}}
\newcommand{\knum}{\mathbb{K}}
\newcommand{\grobner}{Gr\"{o}bner }

\newcommand{\p}[1]{\bm{#1}}
\newcommand{\varxn}{X}
\newcommand{\pointxn}{\bar{\p{x}}}

\newcommand{\rxn}{\mathbb{R}[\varxn]}

\newcommand{\kx}{\mathbb{K}[X]}
\newcommand{\kxn}{\mathbb{K}[\varxn]}

\newcommand{\rc}{\texttt{RealCertify}\xspace}

\newcommand{\maple}{\texttt{Maple}\xspace}

\newcommand{\ignore}[1]{}
\usepackage{todonotes}

\SetKwInput{kwRequires}{Requires}
\SetKwInput{kwEnsures}{Ensures}
\SetKwInput{kwInput}{Input}
\SetKwInput{kwOutput}{Output}
\SetKw{kwGoto}{Go to}
\SetKw{kwBreak}{break}
\SetKw{kwContinue}{continue}

\DeclareMathOperator{\qm}{QM}
\DeclareMathOperator{\quadmod}{QM}

\DeclareMathOperator{\semialgebraic}{S}

\DeclareMathOperator{\algAverkov}{\sf Averkov}
\DeclareMathOperator{\algExtAverkov}{\sf ExtAverkov}
\DeclareMathOperator{\algLasserre}{\sf ExtLasserre}
\DeclareMathOperator{\algPutinar}{\sf Certificate}

\DeclareMathOperator{\lc}{lc}

\DeclarePairedDelimiter{\set}{\lbrace}{\rbrace}



\begin{document}

\begin{frontmatter}



\title{Computing Certificates of Strictly Positive Polynomials in Archimedean Quadratic Modules} %% Article title


\author[1]{Weifeng Shang\corref{cor1}}
\ead{weifengshang@buaa.edu.cn}

\author[2]{Jose Abel Castellanos Joo}
\ead{jose.castellanosjoo@cs.unm.edu}

\author[1,3]{Chenqi Mou}
\ead{chenqi.mou@buaa.edu.cn}

\author[2]{Deepak Kapur}
\ead{kapur@cs.unm.edu}

\cortext[cor1]{Corresponding author}

\affiliation[1]{organization={School of Mathematical Sciences, Beihang University},
    addressline={Beijing 100191},
    country={China}}
\affiliation[2]{organization={Computer Science Department, University of New Mexico},
    addressline={1155 University Blvd SE},
    city={Albuquerque},
    country={United States of America}}
\affiliation[3]{organization={Sino-French Laboratory for Mathematics, Hangzhou International Innovation Institute of Beihang University},
    addressline={Hangzhou 311115},
    country={China}}

\date{}

%% Abstract
\begin{abstract}

This paper addresses the problem of computing
  certificates that witness the membership of strictly positive polynomials in
  multivariate Archimedean quadratic modules. New algorithms for computing such certificates are presented based on (i)
  Averkov's method for generating a strictly positive polynomial for which a
  certificate can be more easily computed, and (ii) Lasserre's method for generating a
  certificate by successively approximating a non-negative polynomial by sums of
  squares. 

  A constructive method based on Averkov's result is given by
  providing details about the parameters; further, his result is extended to
  work on arbitrary subsets, in particular the whole Euclidean space $\rnum^n$,
  producing globally strictly positive polynomials. Then Lasserre's method is extended for monogenic Archimedean quadratic modules and integrated with the extended Averkov construction above to generate
  certificates. 
  
  All the presented methods have been implemented and their effectiveness
  is illustrated. Examples are given on which the existing software package
  \rc appears to struggle, whereas the proposed method succeeds in
  generating certificates. Several situations are identified where an
  Archimedean polynomial does not have to be explicitly included in a set of
  generators of an Archimedean quadratic module. Unlike other approaches for
  addressing the problem of computing certificates, the methods presented are easier to understand as well as to implement.
\end{abstract}

%% Keywords

\end{frontmatter}



\section{Introduction}

Consider a finite polynomial set $G = \{g_1, \ldots, g_s\}$ in the multivariate
polynomial ring $\rxn$. The \emph{quadratic module} generated by $G$ is
$ \qm(G) = \{ \sigma_0 +\sum_{i=1}^{s}\sigma_i \cdot g_i \,|\, \sigma_i $
  is a sum of squares in  $\rxn \mbox{ for } i \!=\!0, \ldots, s\} $
\citet{M08p, P11p}. Quadratic modules are fundamental objects in real algebraic
geometry and are heavily related to \citeauthor{H88u}'s (\citeyear{H88u}) 17th problem and
Positivstellensatz. This paper addresses the problem of computing {\em
  certificates} that witness the membership of polynomials in quadratic modules
in $\rxn$. The certificates here are also referred to as sum-of-squares
representations in the literature. In particular, we present new results and
constructive methods to compute the sum-of-squares multipliers, or certificates,
of strictly positive polynomials in Archimedean finitely generated quadratic
modules.

Several contributions in the field identify conditions whenever representations
in preorderings\footnote{A preordering is a quadratic module closed under
  multiplication.} and quadratic modules exist. Noteworthy are the seminal
results by \citeauthor{S91t}'s (\citeyear{S91t}) and \citeauthor{P93p}'s (\citeyear{P93p}) Positivstellensatz for strictly positive polynomials over compact semialgebraic sets.
However, the above-mentioned results about membership problems do not produce
certificates as their proofs are non-constructive. In the case of Hilbert's 17th
problem, as an example, for a given non-negative polynomial $p$ it is nontrivial
to generate a certificate consisting of rational functions $\frac{a_i}{b_i}$'s
such that $p = \sum_i \left({\frac{a_i}{b_i}}\right)^2$. A constructive solution to the certificate problem for Archimedean preorderings is discussed in \citet{M02a,
  M04s} by including an Archimedean polynomial
\footnote{Given an Archimedean quadratic module $Q$, an \emph{Archimedean
    polynomial} of $Q$ is a polynomial of the form $N - \sum_{i=1}^{n}{x_i^2}$
  with $N \in \mathbb{N}$ that belongs to $Q$.} in the set of generators; an
algorithm is described using \grobner basis techniques to lift positivity
conditions to a simplex and using P\'olya bounds to compute
certificates in \citet{PR01a}.

{\bf Contributions}. This paper presents new results about computing certificates of strictly positive polynomials in Archimedean quadratic modules. The results build upon (i) \citeauthor{A13C}'s (\citeyear{A13C}) method  for generating a strictly positive polynomial over a given compact subset $B$ of $\rnum^n$ for which a membership certificate can be more easily computed than the input polynomial whose certificate is being sought, and (ii) \citeauthor{L07a}'s (\citeyear{L07a}) method for generating a certificate by successively approximating a non-negative polynomial by sums of squares.

A constructive method based on \citeauthor{A13C}'s \citeyear{A13C} result is given
by providing details about the parameters left out there. Further, his result is
extended to work on any subsets, in particular, on $\rnum^n$ (compare \th\ref{lemma:d} and \th\ref{lemma:c} for details). These are discussed in Section~\ref{sec:averkov}. Then in Section~\ref{sec:lasserre} we extend Lasserre's method to reduce any globally strictly positive polynomial by a single polynomial to a sum of squares. These two extensions of Averkov's construction and Lasserre's method are proved in a constructive manner, and they serve as the theoretical and algorithmic foundations of our method for solving the certificate problem. As a by-product, using these two extensions we give a new constructive proof of Putinar's criterion for a quadratic module to be Archimedean (Theorem~\ref{thm:archimedean}) in Section~\ref{sec:criterion}.

To apply the extended Averkov's construction to the certificate problem, there
are two technical problems to solve: construction of a polynomial in the
quadratic module with a bounded non-negative set and a polynomial with a global
lower bound. The former problem is easy when some generator of the quadratic
module has a bounded non-negative set, otherwise we explicitly show how to
construct such a polynomial in the univariate case and use an Archimedean
polynomial in the multivariate case. Unlike the methods in the literature which
depend on the inclusion of Archimedean polynomials in the generating sets, 
we identify the situations where we can avoid inclusion of Archimedean
polynomials and thus in these situations the certificates computed by our method
are with respect to the original generators of the quadratic modules without the
Archimedean polynomials.

After providing solutions to these two technical problems, in Section~\ref{sec:cert} we propose the algorithm (Algorithm~\ref{alg:overall}) for computing certificates of strictly positive polynomials in Archimedean quadratic modules by using the extended Averkov's construction and Lasserre's method. Unlike other approaches for addressing the problem of computing certificates, the proposed algorithm is easier to understand as well as implement.

All the methods discussed in the paper have been implemented and their effectiveness is illustrated.  Examples are given on which \rc, a software package for solving the certificate problem based on semidefinite programming in \citet{M18r}, appears to struggle, whereas the proposed method succeeds in generating certificates. Several examples are given where an Archimedean polynomial does not have to be explicitly included in an input set of generators of an Archimedean quadratic module. 

{\bf Related work}. Among the literature most related to the topic of the paper, the following two publications stand out. In \citet{M18o}, the authors use the extended Gram matrix method to obtain an approximate certificate of the input polynomial in a quadratic module and a refinement step to compute sums of squares to represent the difference of the input polynomial and the approximate certificate (or the ``error term''). The method requires the polynomial set of generators to include an Archimedean polynomial as part of the assumptions. The authors also discuss time complexity for the termination of their algorithm and bounds of the degrees and bit size complexity of the coefficients in the certificates obtained; bounds of the degrees of the certificates rely on the complexity analysis discussed in \citet{N07o}. Their implementation, \rc \citet{M18r}, relies on numerical methods, making them ineffective in certain situations. The performance of the approach proposed in the paper is compared with that of \rc later in Section~\ref{sec:exp}.

More recently, \cite{BM23o} provides degree bounds on the sum-of-square multipliers in the certificate problem for a strictly positive polynomial $f$ in an Archimedean quadratic module $\qm(G)$. This is done in two steps. First, a new polynomial $h \in \qm(G)$ is constructed such that $f - h$ is strictly positive over the hypercube $\left[-1,1\right]^{n}$ by using a modified construction to \citeauthor{A13C}'s (\citeyear{A13C}, [Lemma~3.3]) result. Secondly, the construction in \citet{MS22a} is used to obtain a certificate of $f - h$ in the preordering generated by $1 - \sum_{i=1}^{n}{x_i^2}$ that equals $\qm(1 - \sum_{i=1}^{n}{x_i^2})$. This step relies on the Jackson kernel method, which seems highly complex and also difficult to algorithmize.  Rearranging the previous results, the authors find a certificate of $f$ in $\quadmod(G \cup \{1 - \sum_{i=1}^{n}{x_i^{2}}\})$. This approach also relies on adding an Archimedean polynomial to the generators. The authors do not report any implementation of their approach.

Because of their dependence on requiring an Archimedean polynomial as one of the generators, both of these papers are unable to compute certificates in case the generators of an Archimedean quadratic module do not explicitly include an Archimedean polynomial.

\section{Preliminaries}

Throughout this paper, we use $\rnum$ to denote the field of real numbers, and $\knum$ is a computable subfield of $\rnum$. Here, a computable field is a field for which there is an algorithm for computation of signs of elements of $\knum$ \citet{L70c, S08r}. When talking about effective algorithms in this paper, we restrict ourselves to polynomials over $\knum$ so that all the construction in the algorithms is feasible. For example, the field $\mathbb{Q}$ of rational numbers, a computable subfield of $\rnum$, is frequently used as the base field in the case of computation in real algebraic geometry.

\subsection{\texorpdfstring{Quadratic modules}{Quadratic modules}}

Consider the multivariate polynomial ring $\rxn = \rnum[x_1, \ldots, x_n]$. We denote $\varxn = (x_1,\ldots,x_n)$ and
$\|\varxn\| = \sqrt{x_1^2 + \cdots + x_n^2}$. For any
$\p{\alpha} = (a_1,\ldots, $ $a_n) \in \mathbb{N}^n$, where 
$\mathbb{N}$ is the set of natural numbers, denote
$ \varxn^{\p{\alpha}} = x_1^{a_1} \cdots x_n^{a_n} $. We also use $\pointxn$ to denote a point in $\mathbb{R}^n$.

A polynomial in $\rxn$ is called a \emph{sum of squares} if it can be expressed
as a sum of squares of polynomials in $\rxn$. We denote the set of all sums of
squares in $\rxn$ by $\sum \rxn^2$. Any sum of squares in $\rxn$ is a polynomial
of even total degree and is non-negative over $\rnum^n$.

\begin{definition}\rm 
  A polynomial set $ M \subseteq \rxn$ is called a \emph{quadratic module} in
  $\rxn$ if $M$ is closed under addition, $1 \in M$, and for any $f \in \rxn$
  and $m \in M$, $f^2m \in M$.  A quadratic module $M$ in $\rxn$ is
  \emph{finitely generated} if there exists a finite set
  $G=\{g_1, \dots, g_s\} \subseteq \rxn $ such that
  $ M = \{\sigma_0+\sum_{i=1}^{s}\sigma_i \cdot g_i \,|\, \sigma_i \in \sum
  \rxn^2, i \!=\!0, \ldots, s\} $. In this case, $M$ is said to be
  generated by $G$ and written as $M = \qm(G)$.
\end{definition}

Let $G$ be a polynomial set in $\rxn$. The \emph{non-negative set} of $G$, denoted by $\semialgebraic(G)$, is defined as $\{\pointxn \in \mathbb{R}^n \,|\, g_i(\pointxn) \ge 0 \text{ for all } i=1, \ldots, s\}$. A finitely generated quadratic module $\qm(G)$ in $\rxn$ is said to be \emph{bounded} if the non-negative set $\semialgebraic(G)$ is a bounded set in $\mathbb{R}^n$ and \emph{unbounded} otherwise. Furthermore, a quadratic module in $\rxn$ is said to be \emph{Archimedean} if it contains the polynomial $N - ||\varxn||^2$ for some $N \in \mathbb{N}$, which is called \emph{Archimedean polynomial}. It is clear that Archimedean quadratic modules are bounded. 

\begin{theorem}[Putinar's criterion, \citet{P93p}] \th\label{thm:archimedean}
    For a quadratic module $\qm(G)$ of $\rxn$ with $G = \{g_1,\ldots,g_s\}$, $\qm(G)$ is Archimedean if and only if there exists some $g \in \qm(G)$ such that $\semialgebraic(g)$ is compact.
\end{theorem}

Given a quadratic module $\qm(G)$ in $\rxn$, the fundamental membership problem is to determine whether an arbitrary polynomial $f$ belongs to $\qm(G)$, and this problem has been solved in specific cases. The well-known Putinar's Positivstellensatz about strictly positive polynomials is recalled below. 

\begin{theorem}[Putinar Positivstellensatz, \citet{P93p}]\th\label{thm:Putinar}
    Let $ G=\{g_1,\ldots $ $, g_s\}$ be a polynomial set in $\rxn$  such that $\qm(G) $ is Archimedean and $f \in \rxn$ be a polynomial that is strictly positive on $\semialgebraic(G)$. Then $f \in \qm(G)$.
\end{theorem}

The aim of the present paper is to establish a constructive approach to compute the certificate of $f$ which is strictly positive on $\semialgebraic(G)$, for a given Archimedean quadratic module $\qm(G)$. In this context,  it is generally assumed that the Archimedean property arises from some generator $g \in G$ with bounded $\semialgebraic(g)$. If no such generator exists, we adjoin to $G$ the polynomial $N - ||\varxn||^2$ for a sufficiently large $N$.

\subsection{Averkov's Lemma to aid computing certificates of Positivstellensatz theorems} 

Several (semi-)constructive approaches exist in the literature to compute or construct certificates of strictly positive polynomials in quadratic modules \citet{S05o,N07o,A13C}. A common step in these approaches is the perturbation of the input polynomial $f \in \quadmod(G)$ by translating the positivity of $f$ over $\semialgebraic(G)$ to the positivity of $f - h$ in some compact set $B$ where $h \in \quadmod(G)$. In the following theorems, we restrict to the case $B \neq \emptyset$. When $B = \emptyset$, the statement holds trivially and no further construction is needed. The following result due to Averkov is used in our construction. For self-containedness, the proof in \citet{A13C} is sketched without giving details of the derivation. 

\begin{lemma}[{\citet[Lemma~3.3]{A13C}}]\th\label{lemma:c}
  Let $ G = \{g_1, \ldots, g_s\} $ be a polynomial set in $\rxn$ and $ f $ be a
  polynomial in $\rxn$ such that $f > 0$ on $ \semialgebraic(G) $. Then,
  for any compact subset $ B $ of $ \rnum^n $, a polynomial $ h \in \qm(G) $
  exists such that $ f - h > 0 $ on $ B $.
\end{lemma}

\begin{proof}
  
  Let $ T = \{\pointxn \in B ~|~ f(\pointxn) \leq 0\} $. We use $ a $ as the
  function $ \pointxn \mapsto (g_1(\pointxn), \ldots, g_s(\pointxn)) $ from
  $ \rnum^n $ to $ \rnum^s $. Since $ a(B) $ is compact, there exists
  $ \gamma > 0 $ such that $ a(B) \subseteq (-\infty,2\gamma]^s $. By the
  assumption on $ f $, $ a(T) \cap [0,2\gamma]^s = \emptyset $. As $ a(T) $ and
  $ [0,2\gamma]^s $ are compact, there exists $ \epsilon > 0 $ such that
  $ a(T) \cap [-2\epsilon,\gamma]^s = \emptyset $. Thus, if $ \pointxn \in B $
  and $ g_i(\pointxn) \ge -2\epsilon $ for each $i = 1, \dots, s$, then
  $ f(\pointxn) > 0 $. Consequently,
  $\mu := \min \{f(\pointxn) ~|~ \pointxn \in B\ \mbox{and} \
  g_i(\pointxn) \ge -2\epsilon \mbox{ for } i = 1, \ldots, s\} > 0 $.
    
    Consider the polynomial 
    \begin{equation} \label{eq:aver-g}
        h := \sum_{i=1}^{s}\left(\frac{g_i-\gamma}{\gamma+\epsilon} \right)^{2N} \cdot g_i,
    \end{equation}
    where $N \in \mathbb{N}$ and denote

    $$c(N) := 2\gamma
    \left(\frac{\gamma}{\gamma+\epsilon} \right)^{2N}, \quad C(N) := 2\epsilon
    \left(\frac{\gamma+2\epsilon}{\gamma+\epsilon}\right)^{2N}.$$
    Then for any $ \pointxn \in B $, if $ g_i(\pointxn) \ge -2\epsilon $ for each
    $ i = 1 \ldots, s $, we have
    $ f(\pointxn) - h(\pointxn) \ge \mu -s c(N) $; otherwise if
    $ g_i(\pointxn) < -2\epsilon $ for some $i~(1 \le i \le s)$, we have
    $ f(\pointxn) - h(\pointxn) \ge \min_{\pointxn \in B}f(\pointxn) - s c(N) + C(N)
    $. Since $ c(N) \rightarrow 0 $ and $ C(N) \rightarrow +\infty $ as
    $ N \rightarrow +\infty $, by choosing $ N $ sufficiently large, we deduce
    $ f(\pointxn) - h(\pointxn) > 0 $ for every $ \pointxn \in B $.
\end{proof}

\section{Extending Averkov's results}
\label{sec:averkov}

This section discusses two extensions of Averkov's lemma. Firstly, details about deriving various parameters are given. Secondly, the lemma is generalized to allow arbitrary subsets of $\rnum^n$ for $B$, thus relaxing the requirement that $B$ be compact. These extensions make the generalized lemma to be applicable in wider contexts.

\subsection{\texorpdfstring{Estimation of parameters $\epsilon$, $\gamma$, and $N$}{Estimation of parameters epsilon, gamma, and N}}

The original proof for \th\ref{lemma:c} is semi-constructive: a specific polynomial $ h $ in \eqref{eq:aver-g} is constructed and proved to satisfy the conditions of the lemma; however, parameters like $\epsilon, \gamma$ and $N$ are existentially introduced without specifying proper values. We provide details on the computation of the values of these parameters in the proposition below so that with these values, the constructed polynomial $ h $ in \eqref{eq:aver-g} satisfies the conditions of the lemma. 

\begin{proposition} \th\label{inequality-proof-averkov}
    Let $g_1,\ldots, g_s, f$, and $B$ be as stated in \th\ref{lemma:c} and the parameters $\gamma$, $\epsilon$, and $N$ be such that
    \begin{equation}\label{eq:aver-para}
        \begin{split}       
          \gamma &> \frac{1}{2}\max_{1 \leq i \leq s} \max_{\pointxn \in B} g_i(\pointxn),\\
          \epsilon &< \frac{1}{2} \left|  \max_{\pointxn \in f^{-1}((-\infty,0]) \cap B}   \min_{1 \leq i \leq s} g_i(\pointxn) \right|, \\
          N &> \frac{1}{2}\max\left\{\frac{\log(\mu) - \log(2s\gamma)}{\log(\gamma) - \log(\gamma+\epsilon)}, \frac{\log(M + \mu) - \log(2\epsilon)}{\log(\gamma+2\epsilon) -\log(\gamma+\epsilon)}\right\}, 
        \end{split}
    \end{equation}
    where $M = -\min_{\pointxn \in B}f(\pointxn)$. Then the polynomial $h$ in \eqref{eq:aver-g} satisfies $f - h > 0 $ on $B$. 
\end{proposition}

\begin{proof}
  From the proof of \th\ref{lemma:c}, one can see that $\gamma$ needs to satisfy
  the condition that $2 \gamma \ge g_i(\pointxn)$ for each $i=1, \ldots, s$ and
  any $\pointxn \in B$, and thus the value of $ \gamma$ in \eqref{eq:aver-para}
  suffices.
    
    The condition for $ \epsilon$ is that $ a(T) \cap [-2\epsilon, 2 \gamma]^s = \emptyset $. Then we claim that for any $\pointxn \in T$, there exists $g_i~(1\leq i \leq s) $ such that $ g_i(\pointxn) < -2 \epsilon $. Otherwise, assume that there exists $\pointxn \in T$ such that $ g_i(\pointxn) \geq - 2 \epsilon $ for $i=1, \ldots, s$, then we have $(g_1(\pointxn),\ldots,g_s(\pointxn)) \in a(T) \cap [-2\epsilon, 2\gamma]^s $: a contradiction. From this we know that $-2\epsilon > \min_{ i = 1, \ldots, s} g_i(\pointxn) $ for all $\pointxn \in T$, and thus the value of $ \epsilon$ in \eqref{eq:aver-para} suffices.

    In the end, we prove the inequality for $N$. Let $ M = -\min_{\pointxn \in
      B}f(\pointxn) $. Then from the proof of \th\ref{lemma:c}, one can see that if both $ \mu - s c(N)$ and $-M- s  c(N) + C(N)$ are positive, then $f -h$ is also positive. From $ \mu - s c(N) > 0$, it follows that $ c(N) = 2\gamma \left(\frac{\gamma}{\gamma+\epsilon} \right)^{2N} \!<\! \frac{\mu}{s} $. Therefore, we can set $ 2N > \frac{\log(\mu) - \log(2s\gamma)}{\log(\gamma) - \log(\gamma + \epsilon)} $.  Then with $ c(N) < \frac{\mu}{m} $, we can set $ C(N) > M + \mu $ so that $-M - s c(N) + M + \mu  > 0  $. From this constraint, we obtain the following inequality $N > \frac{1}{2} \frac{\log(M + \mu) - \log(2\epsilon)}{\log(\gamma+2\epsilon) - \log(\gamma+\epsilon)}$. Taking the maximum of both values on the right hand side of the two inequalities above, we have the value for $N$ in \eqref{eq:aver-para}. 
\end{proof}

Plugging the detailed values of the parameters in \th\ref{lemma:d} into Averko-v's method (\th\ref{lemma:c}), Algorithm~\ref{sec:an-algor-descr} below is derived, making Averkov's result constructive.

\begin{algorithm}[h]
    %\DontPrintSemicolon
    \caption{Algorithm for Averkov's construction \\ $ \texttt{certificate} := \algAverkov(G, f, B) $}  
    \label{sec:an-algor-descr}  
    \kwInput{A polynomial set $G = \set{g_1, \dots, g_s} \subseteq \kxn$, a polynomial $f \in \kxn$ such that $f > 0$ over $\semialgebraic(G)$, and a compact set $B \subseteq \mathbb{R}^n$}
    \kwOutput{A sequence of sum-of-squares multipliers $(\sigma_1, \ldots, \sigma_s)$ in $\kxn$ such that $f - \sum_{i=1}^{s}{\sigma_i \cdot g_i} > 0$ over $B$}

    \uIf{$B$ is empty or $\min_{\pointxn \in B}f(\pointxn) > 0$}
    {
        \Return $(0, 0, \ldots, 0)$\;
    }
    \Else
    {
        Choose $\gamma$, $\epsilon$, and $N$ as in \eqref{eq:aver-para}\;
        \Return $\left( \left(\frac{g_1 - \gamma}{\gamma +\epsilon}\right)^{2N}, \ldots, \left(\frac{g_s - \gamma}{\gamma + \epsilon}\right)^{2N}\right)$\;
    }
\end{algorithm}

For actual computation with a compact set $B$ in Proposition~\ref{inequality-proof-averkov}, one can further require that $B$ is a semi-algebraic set represented by polynomial inequalities or is covered by such a semi-algebraic set. In this way, since the parameters in \eqref{eq:aver-g} do not need to be computed exactly,
their estimation reduces to a standard optimization problem. For example, the optimal value of the right-hand side of the second inequality in \eqref{eq:aver-para} can be formulated as
\begin{equation} \label{eq:optimal}
    \begin{aligned}
\text{maximize} \quad & t \\
\text{subject to} \quad & g_i(\pointxn) \ge t, \quad i = 1, \ldots, s, \\
& f(\pointxn) \le 0, \quad \pointxn \in B.
\end{aligned}
\end{equation}

The problem in \eqref{eq:optimal} can be solved by methods like semidefinite programming (SDP) or cylindrical algebraic decomposition (CAD), for example using numerical solvers such as \texttt{Mosek} \citet{mosek} after converting the input into a  standard form. Taking the case when $s=2$, since $ \min \{g_1,g_2\} = \frac{g_1+g_2}{2} - |\frac{g_1-g_2}{2}|$, we add the constraints as $ p^2 = \frac{(g_1-g_2)^2}{2}$ and $p>0$. Then the problem can be restated as

$$    \begin{aligned}
\text{maximize} \quad & t \\
\text{subject to} \quad & \frac{g_1+g_2}{2} - p \ge t, \quad p^2 = \frac{(g_1-g_2)^2}{2} , \\
& p(\pointxn)>0, f(\pointxn) \le 0 \mbox{ for any }\pointxn \in B.
\end{aligned}$$
This technique can then be applied recursively for $s>2$. 

It is also worth mentioning that in our overall algorithm (Algorithm~\ref{alg:overall}) for computing the certificates, we only use the compact set $B$ as a non-negative set $\semialgebraic(g)$ associated with some polynomial $g \in \qm(G)$. The existence of such a polynomial $g$ is guaranteed by the assumption of the Archimedean condition. In particular, the set $B$ is a closed ball if $g = N - ||\varxn||^2$. For a general compact set $B$, our approach requires a positive constant $\gamma$ that witnesses the boundedness of $B$, i.e., $B \subseteq \semialgebraic(\gamma - ||\varxn||^2)$. Then one can use the semialgebraic set $\tilde{B}$ generated by the polynomial $\gamma - ||\varxn||^2$ and compute the parameters as the lower or upper bounds for $\tilde{B}$ instead. 

\begin{example}\rm \label{ex:multiaverkov}
    We illustrate Algorithm~\ref{sec:an-algor-descr} with the example in \citet[5.5.11]{S24a}. Let polynomial set $G = \{g_1,g_2,$ $g_3,g_4\} \subseteq \mathbb{Q}[x_1,x_2,x_3] $, where $g_i = 2x_i - 1 (i=1,\ldots,3), g_4 = 1 - x_1x_2x_3$. The set $\semialgebraic(G) \subseteq [\frac{1}{2}, 4]^3$ is compact and a polynomial $f = -(x_1+3)(x_1-5)-(x_2+3)(x_2-5)-(x_3+3)(x_3-5)$ which is positive on $\semialgebraic(G)$. Then set $B = [-4,4]^3$. First, we compute the minimum of $f$ over $B$, which is $ -24$. Since it is not positive, we follow the else block in the algorithm. The values of $\epsilon, \gamma$ and $N$ computed with \eqref{eq:aver-para} are $3$, $36$, and $21$ respectively.
\end{example}

As noted in \citet[5.5.11]{S24a}, the quadratic module in the above example is non-Archimedean. This example illustrates that Averkov's approach does not rely on the Archimedean condition, and, in fact, does not even require the bounded semialgebraic set $\semialgebraic(G)$. In contrast, the method we present later extends the set $B$ to the entire space and makes essential use of the Archimedean condition. This highlights the critical role that Archimedean property plays in Putinar's Positivstellensatz.

\begin{example}\rm \label{ex:averkov2}
    Now we consider the first example in \citet[Lemma 18]{S25a}. Let $G = \{g_1,g_2,g_3,g_4\} \subseteq \mathbb{Q}[x_1,x_2]$, where $g_1 = x_1$, $g_2 =x_2$, $g_3 = (1-x_1)(1-x_2)$, and $g_4 = 2 - (x_1^2 - x_2^2)$. The set $\semialgebraic(G) = [0,1]^2$ is compact and a polynomial $ f = (x_1 + 1)(2 - x_1) + (x_2 + 1)(2 - x_2)$ which is positive on $\semialgebraic(G)$. Then set $B = \semialgebraic(g_4) = \{\|\pointxn\| \le \sqrt{2}\}$ is compact. First, we compute the minimum of $f$ over $B$, which is $ -\frac{5}{2}$. Since it is not positive, we follow the else block in the algorithm. The values of $\epsilon, \gamma$ and $N$ computed with \eqref{eq:aver-para} are $\frac{1}{3}$, $4$, and $19$ respectively.

\end{example}

\subsection{Generalization of Averkov's Lemma}

This subsection discusses an extension of Averkov's \th\ref{lemma:c} by generalizing the compact set $ B $ to any subset in $ \rnum^{n} $ when the set $ G $ consists of a single generator. Using this result, by setting $B = \mathbb{R}$ in the univariate case, the requirement to add an Archimedean polynomial (whose certificates cannot be easily computed) to the set of generators can be relaxed, thus allowing the computation of the certificates directly with respect to the original generators $G$. Before establishing our result, we first recall the Łojasiewicz inequality at infinity for the polynomial with a compact zero set presented in \cite{G98g}, whose proof employs techniques from algebraic geometry and differential topology.
    
\begin{theorem}[{\citet[Theorem~1]{G98g}}] \th\label{thm:coercive}
Let $F \in \rxn$ be a polynomial of degree $d > 2$ such that the set $F^{-1}(0)$ is compact. Then there exist constants $c, R>0$ such that      
\begin{equation}\label{ineq:coercive}
    |\pointxn^{(d-1)^n - d }| |F(\pointxn)| \ge c \text{\quad for all } |\pointxn|
    > R,
\end{equation}
where $|\pointxn| = \max\{|x_1|,\ldots,|x_n|\}$ is the $L_\infty$ norm of
$\pointxn$.
\end{theorem}

The author of \cite{G98g} points out that no known example achieves the exponent $(d-1)^n - d$ in \th\ref{thm:coercive}. In most cases, the inequality $|\pointxn^{ t } |\cdot | F(\pointxn)| \ge c \text{ for all } |\pointxn|
    > R$ holds for some degree $t < (d-1)^n - d$. This suggests that the degree could potentially be sharpened.
    
\begin{theorem}\th\label{lemma:d}
  Let $g \in \rxn$ be a polynomial with bounded $ \semialgebraic(g) $ and
  $ f \in \rxn$ be strictly positive on $ \semialgebraic(g) $. Then, for any
  subset $ B $ of $ \rnum^n $, there exists a polynomial $ h \in \qm(g) $ such
  that $ f - h > 0$ over $B$.
\end{theorem}

\begin{proof}

  The following proof is constructive and consists of three main steps. First, we construct a modified polynomial $\tilde{g} \in \qm(g)$ such that $\semialgebraic(\tilde{g}) = \semialgebraic(g)$ and $\tilde{g}(\pointxn) \to -\infty$ uniformly as $\|\pointxn\| \to \infty$. In this way, this new polynomial $\tilde{g}$ thereby exhibits behaviors similar to the Archimedean polynomial. Second, we construct a univariate polynomial $H(t)$ with controlled growth and define $h := H(\tilde{g}) \in \qm(g)$. Third, we show that $f - h$ is strictly positive over $B$ by analyzing its behaviors separately on a compact subset and its complement. The former compact case is treated similarly to \th\ref{lemma:c}, while the non-compact case relies on the decay behavior of $\tilde{g}$ at infinity to ensure positivity.

  \textit{Step 1.} Since the non-negative set $\semialgebraic(g)$ of a polynomial $g$
  is bounded, we know that the zero set $g^{-1}(0) \subseteq \semialgebraic(g)$ is both closed and bounded, and thus compact. Without loss of generality, we assume that $\deg(g) > 2$; otherwise, we can use $(1+\| \varxn\|^2) \cdot g$ whose degree is greater than $2$, its zero set is also compact, and a member of $\quadmod((1+\| \varxn\|^2) \cdot g)$ is also a member of $\quadmod(g)$. By \th\ref{thm:coercive} there exist two constants $c_1, R_1>0$
  such that $|g(\pointxn)| \ge c_1|\pointxn|^{d-(d-1)^n} $ for any $\pointxn$ with
  $|\pointxn| > R_1$. It is known that all norms defined on a finite dimensional normed linear
  space are equivalent. Then, there exists $c_2>0$ such that
  $|\pointxn| \ge c_2 \|\pointxn\|$, where $\|\pointxn\| =
  \sqrt{x_1^2+\cdots+x_n^2}$. Meanwhile, with $\semialgebraic(g)$ bounded, there
  exists $R_2 > 0$ such that $|g(\pointxn)| = -g(\pointxn)$ for
  $\|\pointxn\| \ge R_2$. Letting $R = \max\{\frac{R_1}{c_2},R_2\}$ and $c= c_1 c_2^{d-(d-1)^{n}}$,  we have
  $$ -g(\pointxn) = |g(\pointxn)| \ge c_1 |\pointxn|^{d-(d-1)^n} \ge c
  \|\pointxn\|^{d-(d-1)^n} $$ 
  for all $\|\pointxn\| > R $
  since $|\pointxn| \ge c_2\|\pointxn\| > c_2 R \ge c_2 \frac{R_1}{c_2}$.

    Now, let $\tilde{g} = \|\varxn\|^{2\lfloor \frac{(d-1)^n}{2}  \rfloor+2} \cdot g \in
    \qm(g)$. For $\|\pointxn\| > R$, we have \begin{equation}\label{equ:coercive}
        \tilde{g}(\pointxn) \le -c\|\pointxn\|^d.
  \end{equation} 
  Also, for any $K>0$, we know that $\tilde{g} + K \le K - c\|\pointxn\|^d$ for $\|\pointxn\| > R$, which implies that $\semialgebraic(\tilde{g}+K)$ is bounded, for $\semialgebraic(K - c ||x||^d)$, which is contained in a ball of radius $(\frac{K}{c})^{\frac{1}{d}}$, is bounded. Clearly, $\semialgebraic(\tilde{g})$ is bounded. If the set $B$
  is compact, then the condition falls into the case of \th\ref{lemma:c}, and thus it suffices to consider the case when $B$ is unbounded. 
  
  \textit{Step 2.} Consider the closed set $ T = \{\pointxn \in B ~|~ f(\pointxn) \le 0\} $. Since $ \semialgebraic(\tilde{g}) $ is bounded, there exists $ \gamma > 0 $ such that $ \tilde{g}(B) \subseteq (-\infty,2\gamma] $.  If there exists $\pointxn \in T$ such that $ \tilde{g}(\pointxn) \in [0,2\gamma] $, then $ \tilde{g}(\pointxn) \ge 0$. In other words, $\pointxn \in \semialgebraic(\tilde{g})$ but $f(\pointxn) \le 0$, which contradicts the fact that $ f> 0$ on $\semialgebraic(\tilde{g})$. Assign $\epsilon < \frac{1}{2} \left|  \max_{\pointxn} \{   \tilde{g}(\pointxn) ~|~ \pointxn \in T\}\right|$. On the one hand, the polynomial $\tilde{g}$ tends to $-\infty$ when $\|\pointxn\| \to \infty$, then for any fixed point $\bar{\p{y}} \in T $, the set $U_{\bar{\p{y}}} := \{\pointxn \in T ~|~ \tilde{g}(\pointxn) \ge \tilde{g}(\bar{\p{y}})\}$ is compact, thus  $\max_{\pointxn} \{\tilde{g}(\pointxn)~|~ \pointxn \in T\}  = \max_{\pointxn} \{\tilde{g}(\pointxn)~|~ \pointxn \in U_{\bar{\p{y}}} \}$; on the other hand, with $\semialgebraic(g) \subseteq \rnum^n \setminus T$ we have $g(\pointxn) < 0$ for any $\pointxn \in T$. Hence, $\max_{\pointxn} \{   \tilde{g}(\pointxn) ~|~ \pointxn \in T\}$ exists and is $< 0$. Then, we have $\tilde{g}(\pointxn) < -2\epsilon$ for $\pointxn \in T $, and thus $\tilde{g}(T) \cap [-2\epsilon,2\gamma ] = \emptyset$. Consider the univariate polynomial
    $ H(t) := (\frac{t-\gamma}{\gamma+\epsilon})^{2N} \cdot t \in \rnum [t] $, where
    $ N \in \mathbb{N} $ is to be fixed below, and let $c(N)$ and $C(N)$ be
    defined as in the proof of \th\ref{lemma:c}. Then, we have that
    $c(N) \geq H(t)\geq 0$ on $[0,2\gamma]$ and $C(N) \leq -H(t)$ on
    $(-\infty, -2\epsilon)$. Define $ h(\varxn) := H(\tilde{g}(\varxn))$. Then
    $h = (\frac{\tilde{g}-\gamma}{\gamma+\epsilon})^{2N} \cdot \tilde{g}
    \in \qm(g)$. 
    
    \textit{Step 3.} Now we prove that $f - h$ is strictly positive on $B$ with a
    proper choice of $N$. Firstly, we assign a real number $A>R>0$ such that the
    compact set $ D:=\{\pointxn \in \rnum^n ~|~ \|\pointxn\| \le A\} \subset \rnum^n $
    satisfies that $\semialgebraic(g) \subset D$ and
    $\tilde{g}(\pointxn) < -c\|\pointxn\|^d < -\gamma - 2\epsilon$ for any
    $\pointxn \in \overline{D}$, where $\overline{D}$ stands for the
    complement of $D$ in $\rnum^n$.

    (1) For any $ \pointxn \in B \cap D $, we use a similar approach to deal
    with it as in the proof of \th\ref{lemma:c}. If $ \tilde{g}(\pointxn) \ge
    -2\epsilon $, then we have $ f(\pointxn) - h(\pointxn) \ge \mu -
    H(\tilde{g}(\pointxn)) \ge \mu - c(N)$, where $\mu$ is as defined in the
    proof of \th\ref{lemma:c}. Otherwise, $ \tilde{g}(\pointxn) < -2\epsilon $, and we have $ f(\pointxn) - h(\pointxn) \ge f(\pointxn) + C(N)$. Since $B \cap D$ is compact, $f$ has a lower bound $L := \min_{\pointxn \in B \cap D} f(\pointxn) $, and we have $ f(\pointxn) - h(\pointxn) \ge L + C(N)$. Since $ c(N) \rightarrow 0 $ and $ C(N) \rightarrow +\infty $ as $ N \rightarrow +\infty $, we know that $ f(\pointxn) - h(\pointxn) > 0 $ for any $ \pointxn \in D $ when 
    $$N > \frac{1}{2}\max\left\{\frac{\log(\mu) - \log(2\gamma)}{\log(\gamma) - \log(\gamma+\epsilon)}, \frac{\log(L + \mu) - \log(2\epsilon)}{\log(\gamma+2\epsilon) - \log(\gamma+\epsilon)}\right\}.$$
    This conclusion follows from the proof of \th\ref{inequality-proof-averkov}.

    (2) For any $\pointxn \in B \cap  \overline{D}$, we have
    \begin{equation*} 
    \label{sec:extens-averk-meth}
        \begin{split}
            f - h
            &= f - \left(\frac{\tilde{g}-\gamma}{\gamma+\epsilon} \right)^{2N} \cdot \tilde{g} = \tilde{g}^{2k} \left(\frac{f}{\tilde{g}^{2k}} - \tilde{g}^{1-2k}(\frac{\tilde{g}-\gamma}{\gamma+\epsilon})^{2N} \right)\\
            &= \tilde{g}^{2k} \left(\frac{f}{\tilde{g}^{2k}} - \tilde{g} \left(\frac{1-\frac{\gamma}{\tilde{g}}}{\gamma+\epsilon} \right)^{2k}(\frac{\tilde{g}-\gamma}{\gamma+\epsilon})^{2N-2k} \right).
        \end{split}
    \end{equation*}
    Suppose that $f$ is in the form $ f = \sum_{\p{\alpha}}\lambda_{\p{\alpha}} \p{x}^{\p{\alpha}}$. Then we have

    \begin{equation*}
        \begin{split}
            \left|\frac{f(\pointxn)}{\tilde{g}^{2k}(\pointxn)}\right| & \le \frac{\sum_{\p{\alpha}} |\lambda_{\p{\alpha}} \pointxn^{\p{\alpha}}|}{c^{2k}\|\pointxn\|^{2d\cdot k}} \le \frac{\sum_{\p{\alpha}} |\lambda_{\p{\alpha}}| \cdot \|\pointxn\|^{|\p{\alpha}|}}{c^{2k}\|\pointxn\|^{2d \cdot k}} \\
            &=\sum_{\p{\alpha}} \frac{|\lambda_{\p{\alpha}}|}{c^{2k}}\cdot \|\pointxn\|^{|\p{\alpha}|-2 d \cdot k} \le \sum_{\p{\alpha}} \frac{|\lambda_{\p{\alpha}}|}{c^{2k}}\cdot A^{|\p{\alpha}| - 2d \cdot k}
        \end{split}
    \end{equation*}
    for any $\pointxn \notin D$ when $2d \cdot k \ge \deg(f)$. Denote $\Lambda := \sum_{\p{\alpha}} \frac{|\lambda_{\p{\alpha}}|}{c^{2k}}\cdot A^{|\p{\alpha}| - 2d \cdot k}$. With $\tilde{g} < -\gamma - 2\epsilon$, we have
    $$\frac{f}{\tilde{g}^{2k}} - \tilde{g} \left(\frac{1-\frac{\gamma}{\tilde{g}}}{\gamma+\epsilon}\right)^{2k} \left(\frac{\tilde{g}-\gamma}{\gamma+\epsilon}\right)^{2N-2k} > - \Lambda + \gamma\left(\frac{1}{\gamma+\epsilon}\right)^{2k}\left(\frac{2\gamma+2\epsilon}{\gamma+\epsilon}\right)^{2N-2k}$$
    when $N > \frac{\log(\Lambda) + 2k \log(\gamma+\epsilon) - \log(\gamma)}{2\log(2)} +k$. Note that the right-hand side of the above inequality is greater than $0$, which implies that $f - h > 0$.
\end{proof}

\begin{remark}\rm
Since the set $B$ in Lemma~\ref{lemma:c} is compact, it is proved that the polynomial obtained $f-h$ is not only strictly positive on $B$ but it also has a strictly positive minimum on $B$. In fact, the polynomial $f-h$ in Theorem~\ref{lemma:d} can also be proven to have a strictly positive minimum on $B$, even though $B$ here is not necessarily compact: in part (2) of the proof, let $\mu = - \Lambda + \gamma\left(\frac{1}{\gamma+\epsilon}\right)^{2k}\left(\frac{2\gamma+2\epsilon}{\gamma+\epsilon}\right)^{2N-2k} > 0$. Then we have $f - h > \mu \tilde{g}^{2k} \to +\infty$ when $\|\p{x}\| \to \infty$. This property of $f-h$ will be used later in combination with Theorem~\ref{avoid_archimedean_poly} in our algorithm to compute the certificates. 
\end{remark}

 If the polynomial $g$ in the theorem statement already satisfies the inequality \eqref{equ:coercive}, we can directly set $\tilde{g} := g$. Specifically, all Archimedean polynomials satisfy this inequality. 

 \th\ref{lemma:d} implicitly indicates that the polynomial with bounded non-negative set is equivalent to an Archimedean polynomial $N - ||\varxn||^2$ in the context of quadratic modules. This equivalence implies that the boundedness of the non-negative set of the polynomial provides a criterion for determining whether a quadratic module is Archimedean in \th\ref{thm:archimedean}. We shall formally establish this connection at the section \ref{sec:criterion}, by proving \th\ref{thm:archimedean} in a constructive manner.

\begin{example}\rm
    Consider the polynomial $f = 9 - x^3 - y^3 + 3xy -5x^2$ and an Archimedean polynomial $g = 1-x^2-y^2$ in $\mathbb{Q}[x,y]$. One can check that $\{(x, y): f(x, y) \le 0\} \cap \{(x, y): g(x, y) \ge 0\} = \emptyset$, and thus $f$ is strictly positive over $\semialgebraic(g)$.

    Let $B$ in \th\ref{lemma:d} be $\rnum^2$ itself and $\tilde{g} = g$. Then, we assign the parameters $\gamma = \frac{2}{3}$, $\epsilon = \frac{1}{6}$, and $N = 2$ to make $f - (\frac{\tilde{g}-\gamma}{\gamma+\epsilon})^{N} \cdot \tilde{g} $ strictly positive on $\rnum^2$: one can check that $\{(x, y): f(x, y) - (\frac{\tilde{g}(x, y)-\gamma}{\gamma+\epsilon})^{N} \cdot \tilde{g}(x, y) \le 0\} = \emptyset$.
\end{example}

The proof of \th\ref{lemma:d} is a constructive one, that is to say, one can explicitly construct the polynomial $h = (\frac{\tilde{g}-\gamma}{\gamma+\epsilon})^{2N} \cdot \tilde{g} \in \qm(g)$ via construction of the intermediate polynomial $\tilde{g}$ and the computation of the involved parameters $\epsilon$, $\gamma$, and $N$. The values of $\epsilon$ and $\gamma$ are the same as in \eqref{eq:aver-para}, but $N$ needs to satisfy two inequalities in the two cases and is dependent on other parameters like $M$, and $A$ in the proof, and thus its computation is a bit complex. But if we endow $f$ in Theorem~\ref{lemma:d} with an additional property to have a lower bound on $B$, then the proof of Theorem~\ref{lemma:d} will become much more simplified in the following two aspects: (i) The explicit construction of $\tilde{g}$ becomes unnecessary. Instead, the polynomial $g$ directly replaces $\tilde{g}$. (ii) The analysis with the set $D$ can be omitted, so is case (2). This is because, in case (1), for any $\pointxn \in B$, when $g(\pointxn) < -2\epsilon$, the finiteness of $L := \min_{\pointxn \in B} f(\pointxn)$ implies $ f(\pointxn) - h(\pointxn) < L + C(N)$ which covers case (2). 

We formalize these observations as the following \th\ref{cor:lowerbound} and
omit its formal proof. In fact, the simplifications we discuss above due to the
existence of the lower bound make the proof of \th\ref{cor:lowerbound} quite
similar to that of \th\ref{lemma:c}: the polynomial $h$ is in the form $h = (\frac{g-\gamma}{\gamma+\epsilon})^{2N} \cdot g$ and the values of all the parameters are the same as in \eqref{eq:aver-para}. 

\begin{theorem}
\th\label{cor:lowerbound}
    Let $g \in \rxn$ be a polynomial with bounded $ \semialgebraic(g) $ and $ f \in \rxn$ be strictly positive on $ \semialgebraic(g) $. Then, for any subset $ B $ of $ \rnum^n $, if $f$ has a lower bound on $B$, there exists a polynomial $ h \in \qm(g) $ such that $ f - h > 0$ over B.
\end{theorem}

Comparing \th\ref{lemma:d} and \th\ref{cor:lowerbound}, one can see that the
former is more general, while the latter, with an additional assumption, is
easier to implement. In our method for solving the certificate problem, we
mainly use \th\ref{cor:lowerbound} and below we formulate it into
Algorithm~\ref{alg:ExtAverkov}. In fact, our experiments show that using
\th\ref{cor:lowerbound} by pre-construction of another polynomial with a lower
bound usually leads to lower degrees of the computed certificates than using \th\ref{lemma:d} directly.

\begin{algorithm}[h]
    %\DontPrintSemicolon
    \caption{Algorithm for extended Averkov's construction \qquad$ \texttt{certificate} := \algExtAverkov(g, f, B) $}  
    \label{alg:ExtAverkov}
    \kwInput{A polynomial $g \in \kxn$ with bounded $\semialgebraic(g)$, a
      polynomial $f \in \kxn$ which has a lower bound on $B$ and is strictly
      positive over $\semialgebraic(g)$, and a set $B \subseteq \mathbb{R}^n$}
    \kwOutput{A sum-of-squares multiplier $\sigma$ in $\kxn$ such
      that $f - \sigma \cdot g > 0$ over $B$}
    
    \uIf{$B$ is empty or $\min_{\pointxn \in B}f(\pointxn) > 0$}
    { 
        \Return $0$\; 
    } 
    \Else
    {
        Choose $\gamma$, $\epsilon$, and $N$ as in \eqref{eq:aver-para}\;
        \Return $ \left(\frac{g - \gamma}{\gamma +\epsilon}\right)^{2N}$\;
    }
\end{algorithm}

\section{Extending Lasserre's method to a monogenic Archimedean qua-dratic module}
\label{sec:lasserre}

\cite{L07a} gave a method to approximate a non-negative polynomial by using sums of squares, and it shows that every real non-negative polynomial $f$ can be approximated as closely as desired by a sequence of polynomials that are sums of squares. This method is extended to work on a monogenic Archimedean quadratic module. Lasserre's result is recalled below first; this is followed by the proposed extension.

\begin{theorem}[{\citet[Theorem 4.1]{L07a}}]
  \th\label{lasserre} Let $f \in \rxn$ be non-negative with global infimum
  $f^{*} = \inf_{\pointxn \in \rnum^n} f(\pointxn)$, that is,
  $f(\pointxn) \geq f^{*} \geq 0$ for any $\pointxn \in \rnum^{n}$. Then for
  every $\epsilon > 0$, there exists some $r(\epsilon,f) \in \mathbb{N}$ such
  that,
    \begin{equation*}
        f_{\epsilon} = f + \epsilon \sum_{k=0}^{r(\epsilon,f)} \sum_{j=1}^{n} \frac{x_{j}^{2k}}{k!} 
    \end{equation*}
    is a sum of squares.
\end{theorem}

Notice that whenever $ r \ge r(\epsilon,f)$, the polynomial  
$$ f + \epsilon \sum_{k=0}^{r} \sum_{j=1}^{n} \frac{x_{j}^{2k}}{k!} = f_{\epsilon} + \epsilon \sum_{k=r(\epsilon,f)+1}^{r} \sum_{j=1}^{n} \frac{x_{j}^{2k}}{k!}$$ 
is also a sum of squares. The parameter $r$ mentioned can be determined by solving a classical semi-definite programming problem. For details of this method, we refer to \citet{L07a}.

Next, we extend this method to a monogenic quadratic module $\qm(g)$ with a bounded $\semialgebraic(g)$. For a positive polynomial $f$, our objective is to identify an element $h \in \qm(g)$ with $\semialgebraic(g)$ bounded such that $f - h$ is a sum of squares over $\rnum^n$. 

\begin{theorem}
  \th\label{avoid_archimedean_poly} Let $f \in \rxn$ be a positive polynomial
  with a global minimum $f^{*} = \inf_{\pointxn \in \rnum^n} f(x) > 0$ and
  $g \in \rxn$ with $\semialgebraic(g)$ bounded. Then there exist
  $\epsilon > 0$ and $r \in \mathbb{N}$ such that
    \begin{equation}
        \label{eq:globalCeri}
        f_{\epsilon} = f - \epsilon \sum_{k=0}^{r} \sum_{j=1}^{n} \frac{x_{j}^{2k}}{k!} \cdot g
    \end{equation}
    is a sum of squares.    
\end{theorem}

\begin{proof}
  Denote the sum of squares
  $\sum_{k=0}^{r} \sum_{j=1}^{n} \frac{x_{j}^{2k}}{k!} $ by
  $p_r(\varxn)$. Consider $\tilde{g} $ in the proof of \th\ref{lemma:d}, then we show there that for any $K>0$, $\semialgebraic(\tilde{g}+K)$ is bounded, i.e. there
  exists $R>0$ such that
  $\semialgebraic(\tilde{g}+K) \subseteq \{\pointxn~|~ \|\pointxn\|^2 \le
  R\}$. Then, for any $\pointxn$ such that $\|\pointxn\|^2 \le R$, we have
  $p_{r}(\pointxn) \le \sum_{k=0}^{r} \sum_{j=1}^{n} \frac{R^k}{k!} \le ne^R$
  for any $r>0$.
    
    For the sake of simplicity, let $K=1$. Consider the polynomial 
    $$f_{\epsilon} := f - \epsilon p_{r} \cdot \tilde{g} = (f-\epsilon(\tilde{g}+1)p_r) + \epsilon  p_r.$$ 
    Let $M = \max_{\pointxn} (\tilde{g}) + 1$, which is finite, and assign $\epsilon = \frac{f^{*}}{2Mne^R} $. Then we have $h :=f - \epsilon(\tilde{g}+1)p_r \ge f - \frac{f^{*}}{2} > 0$. Applying \th\ref{lasserre} to $h$ with $\epsilon$, we know that there exists some $r(\epsilon, h) \in \mathbb{N}$ such that $h + \epsilon p_{r(\epsilon, h)} = f_\epsilon$ is a sum of squares.
\end{proof}

This constructive proof immediately translates to an algorithmic procedure we formulate as Algorithm~\ref{alg:Lasserre}. Instead of computing $r$ by solving a SDP problem in the algorithm, one can simply increase the value of $r$ iteratively and check whether the resulting $f_{\epsilon} = f - \epsilon p_{r} \cdot \tilde{g}$ becomes a sum of squares. 

\begin{algorithm}[ht]
    %\DontPrintSemicolon
    \caption{Algorithm for finding certificate of globally strictly positive polynomials \qquad $ \texttt{certificate} := \algLasserre(f,g) $}  
    \label{alg:Lasserre}
    \kwInput{A polynomial $g \in \kxn$ with bounded $\semialgebraic(g)$, a polynomial $f \in \kxn$ such that $f > 0$ over $\rnum^n$}
    \kwOutput{A sum-of-squares multiplier $\sigma$ in $\kxn$ such that $f - \sigma \cdot g $ is a sum of squares} 
    {
        \uIf{$f$ is a sum of squares}
        {
            \Return $0$\; 
        } 
        \Else
        {
            Choose $R$ in proof of \th\ref{avoid_archimedean_poly}\; 
            \uIf{$\semialgebraic(g+1)$ is not bounded}
            {
                $c = \|\varxn\|^{2\lfloor \frac{(d-1)^n}{2}  \rfloor+2} $; 
            }
            \Else
            {
                $c = 1$;
            }
             $f^{*} = \min_{\pointxn} f$; $M = \max_{\pointxn}(c \cdot g) + 1$\; $\epsilon = \frac{f^{*}}{2Mne^R}$\; Compute $r$ by solving a SDP problem \;
            \Return $ c \cdot\epsilon \sum_{k=0}^{r} \sum_{j=1}^{n} \frac{x_{j}^{2k}}{k!}  $\;
        }
    }
\end{algorithm}

\begin{example} \rm
    Consider the polynomials $g = 1 - x^4 y^2 - x^2 y^4 + x^2 y^2 - y^6 - x^6$ and $f = 2 + x^4 y^2 + x^2 y^4 - 3x^2 y^2$. The latter is the Motzkin's polynomial plus 1, which is strictly positive over $\rnum^{2}$. Using the technique in \citet[Lemma~1]{B87t}, one can show that $f$ is not a sum of squares.The minimum value of $f$ is $1$ and the maximum value of $g + 1$ is $\frac{217}{108}$. Then the value of $\epsilon$ is $\frac{27}{217 e^3}$. We start with $r = 0$: now $p_r = n $ and $h = f - \epsilon(g + 1)n$, and in fact $f_{\epsilon} = h + \epsilon n$ is verified to be a sum of squares. 
\end{example}

\section{A constructive proof of Putinar's criterion}
\label{sec:criterion}
With the necessary extensions now in place to address the certificate problem, we apply the constructive approach here to prove \th\ref{thm:archimedean}, by computing the certificate of Archimedean polynomial in terms of $\qm(g)$ with bounded $g$. 

\begin{proof}[Proof of \th\ref{thm:archimedean}]
  The implication ($\Rightarrow$) is trivial, now we prove the other direction $(\Leftarrow)$. Since $g \in \qm(G)$ with bounded
  $\semialgebraic(g)$, there exists $R > 0 $ such that
  $\semialgebraic(g) \subseteq \{\pointxn \in \rnum^n ~|~ \|\pointxn\| <R\}$. Then we know that the
  Archimedean polynomial $f := R^2 - ||\varxn||^2 > 0 $ over
  $\semialgebraic(g)$. By \th\ref{lemma:d}, there exists
  $h \in \qm(g) \subseteq \qm(G)$ such that $\tilde{f} = f - h > 0$ over
  $\rnum^n$ with strictly positive global minimum. Then, applying \th\ref{avoid_archimedean_poly} to $\tilde{f}$, we
  know that 
  $\tilde{f}_{\epsilon} := \tilde{f} - \epsilon p_r \cdot g$ is a sum of squares for some $\epsilon> 0$ and $r \in \mathbb{N}$, where $p_r$ is the sum of squares
  $\sum_{k=0}^{r} \sum_{j=1}^{n} \frac{x_{j}^{2k}}{k!} $.
Now we have \begin{equation}\label{eq:certificate}
        f = (h+\epsilon p_r) \cdot g + \tilde{f}_\epsilon . 
    \end{equation}
    Note that $\epsilon p_r$ is also a sum of squares in
    $\rxn$. Therefore, \eqref{eq:certificate} presents the certificate of
    $R^2 - ||\varxn||^2$ in term of $\qm(G)$, and thus the quadratic module
    $\qm(G)$ is Archimedean.
\end{proof}

\section{Computing certificates using the extended Averkov construction and Lasserre's method}
\label{sec:cert}

In this section, the extended Averkov construction and Lasserre's method are integrated to compute certificates in the general case. There are two issues that need to be addressed; this is discussed in the next subsection. An algorithm to compute certificates is given in the subsection to follow.

\subsection{Constructing a bounded polynomial with a global lower bound for applying the extended Averkov's method}

Two issues need to be addressed to apply the extended Averkov's method. Per \th\ref{cor:lowerbound}, a polynomial $g \in \qm(G)$ with a bounded $\semialgebraic(g)$ must be constructed; further, its global lower bound needs to be computed.

The univariate and multivariate cases are handled separately for computing the above $g$. If there exists a generator $g_i \in G$ such that $\semialgebraic(g_i) $ is bounded, then use $g_i $ as $g$. Otherwise, all $\semialgebraic(g_1), \ldots, \semialgebraic(g_s) $ are unbounded, and in the multivariate case we assign $g$ to be the Archimedean polynomial $g := N - ||\varxn||^2$  with a proper choice of $N$ such that $f > 0$ on $\semialgebraic(g)$. In the univariate case, however, there must exist two generators $ g_i $ and $ g_j ~(i \ne j) $ such that the leading coefficients $ \lc(g_i) $ and $ \lc(g_j) $ have opposite signs and the degrees of $ g_i $ and $ g_j $ are odd. Let $g := c_i \cdot g_i + c_j \cdot g_j$ where 
\begin{equation}\label{eq:1gen}
    \left\{
        \begin{array}{ll}
            c_i := \frac{1}{|\lc(g_i)|}, c_j := \frac{(x-a)^{\deg(g_i)-\deg(g_j)}}{|\lc(g_j)|},  & \mbox{if~} \deg(g_i) > \deg(g_j), \\
            c_i := \frac{x^2}{|\lc(g_i)|}, c_j := \frac{(x-a)^2}{|\lc(g_j)|}, & \mbox{if~} \deg(g_i) = \deg(g_j),
        \end{array}
    \right.
\end{equation}
The constant $ a $ is selected such that $ \lc(g) < 0 $. One can see that the choice of $c_i, c_j$ forces the cancellation of the leading odd terms of $ g_i $ and $g_j$, hence $\deg(g)$ is even. Clearly, we have $ \semialgebraic(G) \subseteq \semialgebraic(g) $ and $ g \in \qm(G) $. 

The next step is to construct a new polynomial $ \tilde{f} : = f - \sigma \cdot g$
such that $\tilde{f} > 0$ on $\semialgebraic(g)$, where $\sigma$ is the
certificate returned by $\algAverkov(\{g\}, f, \semialgebraic(g))$. At this
step, $\tilde{f}$ and $g$ as constructed above satisfy the conditions of
\th\ref{lemma:d}; however, $\tilde{f}$ still misses a critical property of possessing a lower bound, in particular a global lower bound if $B$ in \th\ref{cor:lowerbound} is $\rnum^n$. The following proposition addresses this problem.

\begin{proposition}\th\label{prop:lowerBound}
    Let $ f,g \in \rxn$ be two polynomials with $\semialgebraic(g)$ bounded and $ f > 0 $ on $\semialgebraic(g)$. Then there exists a sum of squares $\delta$ such that $f - \delta \cdot g$ is strictly positive on $\semialgebraic(g)$ and has a lower bound over $\rnum^n$.
\end{proposition}
    
\begin{proof}
    Let $m= \deg(f) + (\deg(g) - 1)^n - \deg(g)$ and consider the sum of squares
    \begin{equation}\label{eq:lowerBound}
      \delta = c \|\pointxn\|^{2\lfloor \frac{m}{2} \rfloor + 2}, \mbox{ where }c < \frac{1}{2} \min_{\pointxn \in \semialgebraic(g)}\frac{f(\pointxn)}{\|\pointxn\|^{2\lfloor \frac{m}{2} \rfloor + 2} g(\pointxn)}.
    \end{equation}
    Next, we prove that $\bar{f} := f - \delta \cdot g$ is strictly positive on $\semialgebraic(g)$ and has a lower bound over $\rnum^n$. Suppose that $f$ is of the form $f = \sum_{\p{\alpha}}\lambda_{\p{\alpha}} \p{x}^{\p{\alpha}}$ with $\lambda_{\p{\alpha}}  \in \rnum$. Then for any $\pointxn \in \semialgebraic(g)$, we have $\frac{f(\pointxn)}{\|\pointxn\|^{2\lfloor \frac{m}{2} \rfloor + 2} g(\pointxn)} > c $ and thus we have $ \bar{f}(\pointxn) = f(\pointxn) - \delta(\pointxn)g(\pointxn) > 0$. 
    
    Analogously, in the first two paragraphs of the proof of \th\ref{lemma:d}, there exist $l > 0$ and $R \ge 1$ such that $ |  \|\pointxn\|^{(\deg(g)-1)^n} g(\pointxn) |\ge l \|\pointxn\|^{\deg(g)}$ for $\|\pointxn\| > R$. Then for any $ \|\pointxn\| > R $, we have 
    \begin{equation*}
        \begin{split}
            | \|\pointxn\|^{2\lfloor \frac{m}{2} \rfloor + 2} g(\pointxn) | \ge | \|\pointxn||^{m+1} g(\pointxn)| \ge
            l \|\pointxn\|^{\deg(g) - (\deg(g)-1)^n + m+1  } = l\|\pointxn\|^{\deg(f)+1},
        \end{split}
    \end{equation*}
    and $|f|  \le \sum_{\p{\alpha}} |\lambda_{\ p{\alpha}} \p{x}^\alpha| \le \sum_{\p{\alpha}} |\lambda_{\p{\alpha}}|\cdot \|\pointxn\|^{|\p{\alpha}|}$. 
    
    Now consider the rational function 
    $$\left|\frac{f}{\delta \cdot g}\right| \le \left|\frac{f}{c\|\pointxn\|^{2\lfloor
          \frac{m}{2} \rfloor + 2} \cdot g}\right| \le \sum_{\p{\alpha}}
    \frac{|\lambda_{\p{\alpha}}|}{c\cdot l} \|\pointxn\|^{|\p{\alpha}| - \deg(f) - 1}.$$
    Note that $|\p{\alpha}| - \deg(f) - 1 < 0 $ for any $\p{\alpha}$. When $\|\pointxn\| \to \infty $, we have $\|\pointxn\|^{|\p{\alpha}| - \deg(f) - 1} \to 0$, and thus there exists $M> 0$ such that $|\frac{f}{\delta \cdot g}| \le 1$ for any $\pointxn$ such that $\|\pointxn\| \ge M$. In other words, for any $\pointxn \in \overline{\semialgebraic(g)} \cap \{\pointxn~|~\|\pointxn\| \ge M\} $, we have $\frac{f}{\delta \cdot g}(\pointxn) - 1 < 0$. Note that $\pointxn \not \in \semialgebraic(g)$, and thus $g(\pointxn) < 0$. Then we have $ \bar{f}(\pointxn) = f(\pointxn) - \delta(\pointxn)g(\pointxn) > 0$. This indicates that the set $\semialgebraic(-\bar{f})$ is contained in $\{\|\pointxn\| \ge M\}$ and bounded, and thus $\bar{f}$ has a lower bound over $\rnum^n$.
\end{proof}

\subsection{Algorithm for solving the certificate problem}

Integrating the extended Averkov construction and Lasserre's method, the result is the following algorithm for computing the certificates in Archime-dean quadratic modules, which are later formulated as Algorithm~\ref{alg:overall}. 

\begin{enumerate}
    \item Find or construct $g \in \quadmod(G)$ such that $\semialgebraic(g)$ is bounded, together with its certificates w.r.t. $\quadmod(G)$.

    \item Apply Averkov's method (Algorithm~\ref{sec:an-algor-descr}) for the subset $B = \semialgebraic(g)$ to construct a polynomial $\tilde{g} \in \quadmod(G)$, together with its certificates w.r.t. $\quadmod(G)$, such that $\tilde{f} := f - \tilde{g} > 0$ over $\semialgebraic(g)$.

    \item If $\tilde{f}$ does not have a lower bound over $\mathbb{R}^n$
      (otherwise we can skip the construction here and set $\hat{f} =
      \tilde{f}$), construct a sum of squares $\delta$ as in \eqref{eq:lowerBound}
      such that, as shown by \th\ref{prop:lowerBound}, $\hat{f} := \tilde{f} - \delta \cdot g \in \quadmod(g)$ is strictly positive on $\semialgebraic(g)$ and has a lower bound over $\rnum^n$. 

    \item Applying \th\ref{cor:lowerbound}~(Algorithm~\ref{alg:ExtAverkov}) to $\hat{f}$ for $B = \rnum^n$, construct a polynomial $\hat{g} \in \quadmod(G)$ such that $\hat{f} - \hat{g} > 0$ over $\mathbb{R}^n$. 
 
    \item In the univariate case when $n=1$, $\hat{f} - \hat{g}$ is already a sum of squares itself and we have
    \begin{equation*} 
    \label{item:otherw-find-posit}
        f = \underbrace{\hat{f}-\hat{g}}_{\text{sum of squares}} + \underbrace{\hat{g} + \delta \cdot g + \tilde{g}}_{\in \quadmod(g)}  \in \quadmod(G).
    \end{equation*}
    Combining the certificates of $\hat{g}$, $g$, and $\tilde{g}$ w.r.t. $\qm(G)$ will give those of $f$ w.r.t. $\qm(G)$. 

    \item Otherwise in the multivariate case, we apply the extended Lasserre's method  (Algorithm~\ref{alg:Lasserre}) to $\hat{f}$ to construct $\epsilon$ and $p_r$ such that $\hat{f} - \epsilon p_r \hat{g}$ is a sum of squares. In this case,
    \begin{equation*} 
    \label{item:otherw-find-posit-other}
        f = \underbrace{\hat{f}-\epsilon p_r \cdot \hat{g}}_{\text{sum of squares}} + \underbrace{\epsilon p_r \cdot \hat{g} + \delta \cdot g + \tilde{g}}_{\in \quadmod(g)}  \in \quadmod(G),
    \end{equation*}
    and the cerficates of $f$ w.r.t. $\qm(G)$ also follow from combination of those of $\hat{g}$, $g$, and $\tilde{g}$ w.r.t. $\qm(G)$.
\end{enumerate}

\begin{algorithm}[!htp]
    \caption{Algorithm for solving the certificate problem for strictly positive polynomials \quad $  \texttt{certificates} := \algPutinar(G, f) $}
    \label{alg:overall}
    \kwInput{A polynomial set $G = \set{g_1, \dots, g_s} \subseteq \kxn$ and a polynomial $f \in \kxn$ such that $f > 0$ over $\semialgebraic(G)$}  
    \kwOutput{A sequence $(\sigma_0, \sigma_1, \dots, \sigma_{s+1})$ of sum-of-squares multipliers such that $f = \sigma_0 + \sum_{i=1}^{s}{\sigma_i \cdot g_i} + \sigma_{s+1} \cdot (N - ||\varxn||^2)$} 

    \For{$i=1, \ldots, s+1$}
    {
        $\sigma_i := 0$; $c_i := 0$\;
    }
    \eIf{$\exists g_i \in G \mbox{ such that } \semialgebraic(g_i)$ is bounded}
    {
        $g := g_i; c_i := 1$\; 
    }
    {
        \eIf{$n > 1$}
        {
            $g := N - ||\varxn||^2$; $c_{s+1} := 1$\;
        }
        {
            Find $g_i$ and $g_j$ of odd degrees and leading coefficients of opposite signs\;
            Construct $c_i$ and $c_j$ as in \eqref{eq:1gen}\;
            $g := c_i \cdot g_i + c_j \cdot g_j$\;
        }
    }   

    $(\sigma_1, \ldots, \sigma_s) := \algAverkov(G, f, \semialgebraic(g))$\;
    $\tilde{f} := f - \sum_{i=1}^s \sigma _i g_i$\;
    \uIf{$\tilde{f}$ does not have a global lower bound\label{line:noLower}}
    {
        Construct $ \delta $ as in \eqref{eq:lowerBound};
        $\tilde{f} := \tilde{f} - \delta \cdot g$\label{line:delta}\;
    }
    $p := \algExtAverkov(g, \tilde{f}, \mathbb{R}^n)$\label{line:extAverkov}\;

    \eIf{$n = 1$}
    {
        $\sigma_i := \sigma_i + (p + \delta) c_i$; $\sigma_j := \sigma_j + (p + \delta) c_j$\;
        $\sigma_0 := f - \sum_{i=1}^{s}{\sigma_i \cdot g_i}$\;
    }
    { 
        $ q := \algLasserre(\tilde{f} - p \cdot g,g)$ \;
        \eIf{$c_{s+1} = 1$}
        {
            $ \sigma_{s+1} = (q+p+\delta)c_{s+1} $ \;
        }
        {
            $\sigma_i = \sigma_i + (p+q+\delta)c_i$\;
        }
            $\sigma_0 := f - \sum_{i=1}^{s+1}{\sigma_i \cdot g_i}$\;
    }

    \Return $(\sigma_0, \sigma_1, \dots, \sigma_{s+1})$\;
  
\end{algorithm}

\begin{example} \label{sec:putinar-cert-problem_example} 
\rm

Let us illustrate Algorithm~\ref{alg:overall} with the univariate quadratic module $\qm(g_1,g_2)$, where 
$$g_1 =x(x - \frac{1}{2})(x - 1)^2(x - 2), \qquad g_2=-x(x - 1)(x - 2).$$
Then, we have $\semialgebraic(g_1,g_2) = \{0,1,2\}$. The polynomial 
$$f = -26x^7 + 13x^6 + 87x^5 + 49x^4 - 464x^3 + 1512x^2 - 2211x + 1050$$
is strictly positive on $ \semialgebraic(g_1,g_2)$, and thus $f \in \qm(g_1,g_2)$ by Putinar's Positivstellensatz.

With neither $\semialgebraic(g_1)$ nor $\semialgebraic(g_2)$ bounded, we construct 
$$ g :=  g_1 + x^2 \cdot g_2  = -\frac{1}{2}x(x - 1)(x - 2)(3x - 1)$$ 
in $\qm(g_1,g_2) $ using \eqref{eq:1gen}, and now we have $ \semialgebraic(g) = [0,\frac{1}{3}] \cup [1,2]$. Since $ f > 0 $ on $ \semialgebraic(g) $, we can skip the Averkov's method (\th\ref{lemma:c}). Now, construct the polynomial $\delta = 20 x^4$ by using \eqref{eq:lowerBound} in line~\ref{line:delta} of Algorithm~\ref{alg:overall}. The polynomial $\tilde{f} = f - \delta \cdot g$ has a global lower bound $ -159 $. In line~\ref{line:extAverkov}, by using Algorithm~\ref{alg:ExtAverkov}, we have $ \bar{f} := \tilde{f} - \sigma \cdot g > 0 $ on $\rnum$ so it is a sum of squares in $\mathbb{Q}[x]$, where
$$ \sigma = \left(\frac{25x(x-1)(x-2)(3x-1)}{29} +  \frac{19}{29} \right)^{22}.$$  
Consequently, Algorithm~\ref{alg:overall} terminates with the certificate $f = (\delta+\sigma) \cdot g_1 + (\delta+\sigma) x^2 \cdot g_2 + \tilde{f}$. 

To provide a comparison with the alternative method that does not construct the polynomial $\tilde{f}$ with a global lower bound in lines~\ref{line:noLower}--\ref{line:delta} but uses \th\ref{lemma:d} directly, we present the certificates computed by this alternative method. We have $ \bar{f} := f - \bar{\sigma} \cdot g > 0 $ on $\rnum$ so it is a sum of squares in $\kx$, where
$$ \bar{\sigma} = \left(\frac{50x(x-1)(x-2)(3x-1)}{43} +  \frac{38}{43} \right)^{52}.$$ 
Finally, writing $g$ in terms of $\{ g_1,g_2\}$, we have the certificates of $ f $ w.r.t. $\{ g_1,g_2\}$: $ f = \sigma \cdot g_1 + \sigma x^2 \cdot g_2 + \bar{f}$ and the certificates computed in this way have higher degrees than those by Algorithm~\ref{alg:overall}. 
\end{example}

\begin{example}
    \rm
    Let us consider the another example in \citet[Lemma 18]{S25a}. Let the quadratic module $\qm(G) = \qm(g_1,g_2,g_3,g_4) \subseteq \mathbb{Q}[x,y]$, where $g_1 = x, g_2 = y, g_3 = (1-x)(1-y) \text{ and } g_4 = 2 - (x + y).$ We have $\semialgebraic(G) = [0,1]^2$. Let $$ f = (x + 1)(2 - x) + (y + 1)(2 - y), $$ which is strictly positive on $\semialgebraic(G)$, and thus $ f \in \qm(G)$ by Putinar's Positivstellensatz.

    Now consider the first step to construct the polynomial $g$. Since $ 2 - x = g_4 + g_1 \in \qm(G)$, we have $(2-x)x = \frac{x^2}{2} (2-x) + 2(1-\frac{x}{2})^2 x \in \qm(G)$. Then $ 4 - x^2 = 2(2-x) + (2-x)x \in \qm(G)$. Similarly, we have $ 4 - y^2 \in \qm(G)$. To sum up, we assign $ g =  8 - x^2 - y^2 = (2+\frac{x^2}{2}+2(1-\frac{x}{2})^2) \cdot g_1 +(2+\frac{y^2}{2}+2(1-\frac{y}{2})^2) \cdot g_2+(4 + \frac{x^2}{2} + \frac{y^2}{2}) \cdot g_4 \in \qm(G)$, which is an Archimedean polynomial.

    Then, for the subset $B = \semialgebraic(g) = \{\|\pointxn\| \le 2\sqrt{2}\}$ we construct sums of squares $\sigma_i~ (i =1, \ldots,4)$ by using Averkov's method, such that $\tilde{f} = f - \sum_{i=1}^{4}\sigma_i \cdot g_i > 0 $ over $B$, where $$ \sigma_i = \left( \frac{10 }{51}\cdot g_i - \frac{50}{51}\right)^{4}.$$ 

    Since $\tilde{f} $ has no lower bound, we construct the polynomial $\delta = \frac{(x^2+y^2)^5}{6000}$ by using \eqref{eq:lowerBound} in line~\ref{line:delta} of Algorithm~\ref{alg:overall}. The polynomial $\tilde{f} = \tilde{f} - \delta \cdot g$ is already sums of squares, hence the algorithm terminates. Finally, we have the certificate of $f$ in $\qm(G)$: $f = \sum_{i=1}^4 \sigma_i \cdot g_i + \delta \cdot g + \tilde{f}$.

    In general, the polynomial $\tilde{f}$ at line~\ref{line:delta} is not sums of squares. For example, when construct $\delta = \frac{(x^2+y^2)^5}{40000}$, the polynomial $\tilde{f}$ has a global lower bound of $-6$. In line~\ref{line:extAverkov}, by using Algorithm~\ref{alg:ExtAverkov}, we have $ \bar{f} := \tilde{f} - \sigma \cdot g > 0 $ on $\rnum^2$, and is also a sums of squares, where
$$ \sigma = \left(\frac{10}{47}  \cdot g -  \frac{40}{47} \right)^{6}.$$  
Consequently, Algorithm~\ref{alg:overall} terminates with the certificate $f = \sum_{i=1}^4\sigma_i \cdot g_i + (\delta + \sigma) \cdot g + \bar{f}$. 
    
\end{example}

\section{Experiments}
\label{sec:exp}

This section reports the experimental results with our prototypical implementations \footnote{The implementation can be found using the following link: \url{https://github.com/typesAreSpaces/StrictlyPositiveCert}} in the Computer Algebra System \maple to compute certificates of strictly positive polynomial over bounded semialgebraic sets on a variety of problems. We implemented the extended Lassere's method discussed in Section~\ref{sec:lasserre} using the \texttt{Julia} programming language \footnote{The implementation is hosted in GitHub and can be accessed using the following link: \url{https://github.com/typesAreSpaces/globalStrictPositive}}. We also compare our results with \rc\cite{M18r}, a \maple package that computes certificates of strictly positive elements in quadratic modules using a semidefinite programming approach. All experiments were performed on an M1 MacBook Air with 8GB of RAM running macOS.

\subsection{Univariate certificate problems}

\subsubsection{Certificates of strictly positive polynomials}

The objective of this benchmark is to compare the quality of certificates and the time required by \rc \footnote{We use the command \texttt{multivsos\_interval(f, glist=generators, gmp=true)}.} and our implementation. This benchmark uses as generators the polynomial set of the form $\{p, -p\}$ where $p := \sum_{i=1}^{2n}{(x-a_i)}$ and computes certificates of strictly positive polynomials of the following kinds: a linear polynomial of the form $x - a$ (\texttt{left\_strict} in Table~\ref{tab:1}), a linear polynomial of the form $-(x-a)$ (\texttt{right\_strict}),  the product of the generators $-p^2 + 35$ (\texttt{lifted\_prods}), and the Archimedean polynomial added to the generators for \rc plus $10$ (\texttt{arch\_poly}). 

We produced 100 tests for each of these kinds of problems, and Table~\ref{tab:1} shows the results obtained. Each problem contains two rows: the first row uses \rc adding an Archimedean polynomial to the set of generators\footnote{The \rc package uses the algorithms in \citet{M18o} which require an explicit assumption about the Archimedeaness of the quadratic module; for this reason we extend the set of generators with an Archimedean polynomial as otherwise \rc does not find certificates.}; the second describes the results obtained by our implementation using the original generators for each benchmark problem. To compare the quality of the certificates obtained, we gather the degrees of the sum-of-squares multipliers obtained. We report these in the second column ``Max Poly Degrees'' collecting the maximum value of all the tests, the average and median. The fourth column ``Time'' reports the time in seconds for the average, median, and total of the 100 tests. Finally, the last column reports the number among 100 tests, for which certificates were computed.

\begin{table}[htpb]
    \centering
    \resizebox{\columnwidth}{!}
    {
    \begin{tabular}{c|c|ccc|ccc|c}
        \hline
        Problem & Implementation &
        \multicolumn{3}{c|}{Max Poly Degrees}&
        \multicolumn{3}{c|}{Time (seconds)}&
        Successful\\
        \cline{3-8}
        & & Max & Avg & Median & Avg & Median & Total & runs \\
        \hline
        \multirow{2}{*}{\texttt{left\_strict}}
            & \rc & 6 & 5.573 & 6 & 0.039 & 0.037 & 3.458 & 89/100 \\
        \cline{2-9}
            & Ours & 8 & 6.62 & 7 & 0.05 & 0.05 & 5.045 & 100/100 \\
        \hline
        \multirow{2}{*}{\texttt{right\_strict}}
            & \rc & 8 & 5.5 & 6 & 0.045 & 0.04 & 3.597 & 80/100 \\
        \cline{2-9}
            & Ours & 10 & 6.84 & 7 & 0.065 & 0.058 & 6.505 & 100/100 \\
        \hline
        \multirow{2}{*}{\texttt{lifted\_prods}}
            & \rc & 4 & 4.0 & 4 & 0.014 & 0.01 & 1.365 & 100/100 \\
        \cline{2-9}
            & Ours & 8 & 5.12 & 5 & 0.162 & 0.03 & 16.168 & 100/100 \\
        \hline
        \multirow{2}{*}{\texttt{arch\_poly}}
            & \rc & 8 & 5.236 & 6 & 0.045 & 0.039 & 2.459 & 55/100 \\
        \cline{2-9}
            & Ours & 10 & 6.74 & 7 & 0.088 & 0.073 & 8.772 & 100/100 \\
        \hline
    \end{tabular}
    }
    \caption{Computing certificates of four kinds of strictly positive univariate polynomials}
    \label{tab:1}
\end{table}

As the table shows, even though our implementation runs slower than \rc, it is able to find certificates for all the examples in the benchmark; in contrast, \rc cannot compute certificates for some of these examples. Additionally, our implementation can compute certificates using only the elements of the original set of generators given, whereas \rc requires a polynomial to witness the Archimedeaness of the quadratic module. The Archimedean polynomial is always a member of compact quadratic modules in the univariate case, and the resulting certificates by \rc would include an additional element in the set of generators for which certificates need to be computed explicitly if one wants certificates in the original generators.

\subsubsection{Comparison of degree of certificates}

For the second benchmark, we use a parameterized monogenic generator $(1-x^2)^{k}$ where $k$ is an odd number. It is well known in the literature \citet{S96c, V00p} that the degree of the certificates depends on the degree of the generators, as well as the separation of the zeros of the input polynomial and the endpoints of the semialgebraic set associated with the set of generators. We show with our experimental results with this kind of benchmark that our approach and heuristics can effectively handle this problem.

The benchmark evaluates the degree obtained by our implementation and compares it with \rc. 
The semialgebraic set for this generator is the interval $[-1, 1]$. We use different values of $\epsilon > 0$ and $k$ of the strictly positive polynomial $(1 + \epsilon) + x $ over $[-1, 1]$ to report the degrees of the certificates obtained in $\quadmod((1-x^2)^{k})$ and the time (in seconds) taken in Table~\ref{tab:2} below. Additionally, we provide complexity estimates for this benchmark from \citet{S96c}. We notice both implementations produce certificates within the theoretical bounds.

\begin{table}[ht]
    \centering
    \small
    \begin{tabular}{c|c|cc|cc|cc}
        \hline
        \multirow{2}{0.5cm}{\makecell{$\epsilon$}} &
        \multirow{2}{0.5cm}{\makecell{$k$}} &
        \multicolumn{2}{c|}{\rc} &
        \multicolumn{2}{c}{Ours} 
        & \multicolumn{2}{c}{Theoretical bounds}\\
        \cline{3-8}
        &   & Degree & Time & Degree & Time &
        Lowerbound & Upperbound \\
        \hline
        \multirow{5}{0.5cm}{\makecell{$\frac{1}{2}$}}
        &13 & 26 & 0.094 & 26 &  0.02 & 2 & 94\\
        &15 & 30 & 0.215 & 30 & 0.02 & 2 & 106\\
        &17 & 34 & 0.247& 34 & 0.021 & 2 & 118\\
        &19 & 38 & 0.312 & 38 & 0.022 & 2 & 130\\
        &21 & 42 & 0.361 & 42 & 0.023 & 2 & 142\\
        \hline

        \multirow{5}{0.5cm}{\makecell{$\frac{1}{3}$}}
        & 13 & 26 & 0.181 & 26 & 0.02 & 6 & 128\\
        & 15 & 30 & 0.182 & 30 & 0.02 & 6 & 146\\
        & 17 & 34 & 0.252 & 34 & 0.022 & 6 & 162\\
        & 19 & 38 & 0.319 & 38 & 0.024 & 6 & 180\\
        & 21 & 42 & 0.389 & 42 & 0.026 & 6 & 196\\
        \hline
    \end{tabular}
    \caption{Computing certificates of strictly positive univariate polynomials with separation of $\epsilon = \frac{1}{2}$ and $\frac{1}{3}$ from endpoints}
    \label{tab:2}
\end{table}

In this benchmark, \rc does not need an Archimedean polynomial to compute certificates. In general, our implementation is faster than \rc and obtains certificates with the same degree.

\subsection{Certificates of globally strictly positive multivariate polynomials}

This benchmark is used to test our implementation of the extended Lasse-rre's method discussed in Section~\ref{sec:lasserre}. The quadratic module in consideration is generated by the polynomial $g := 10 - x^2 - y^4 - z^8$. The quadratic module $\quadmod(g)$ is Archimedean as it contains the Archimedean polynomial 
$14 - x^2 - y^2 - z^2$; its certificates are
\begin{equation*} 
  \begin{split}
    &14 -x^2 - y^2 - z^2 =\frac{3854 x^2}{44027}+\frac{3854
      y^4}{44027}+y^2+\left(1-y^2\right)^2+\frac{80476924
      58107 z^8}{1282595968637408}\\
    &+\frac{273
      z^6}{26240}+2z^4+\frac{123}{160}
      \left(z-\frac{55 z^3}{82}\right)^2 
    +\frac{910375997}{5070233600} \left(z^2-\frac{591644335
      z^4}{910375997}\right)^2\\
    &+\frac{49514}{44027}\left(\frac{220135 z^4}{3168896}-\frac{12459641
      z^2}{15844480}+1\right)^2 +\left(1-z^4\right)^2+ \left(\frac{3854}{44027} + 1\right) \cdot g.
  \end{split}
\end{equation*}

For the input polynomials, we use two strictly positive polynomials by adding a positive constant to the Motkzin and Robinson polynomials. These polynomials are, respectively, of the form $x^6 + y^4z^2 + y^2z^4 - 3x^2y^2z^2 + \frac{1}{d}$ and $x^4y^2 + y^4z^2 + x^2z^4 - 3x^2y^2z^2 + \frac{1}{d}$ where $d$ is a positive constant. We encoded the problem in \rc and used several values for its configuration parameter \texttt{relaxorder}; however, the package cannot find certificates for this benchmark.

\begin{figure}[ht!]
  \centering
  \includegraphics[width=\columnwidth]{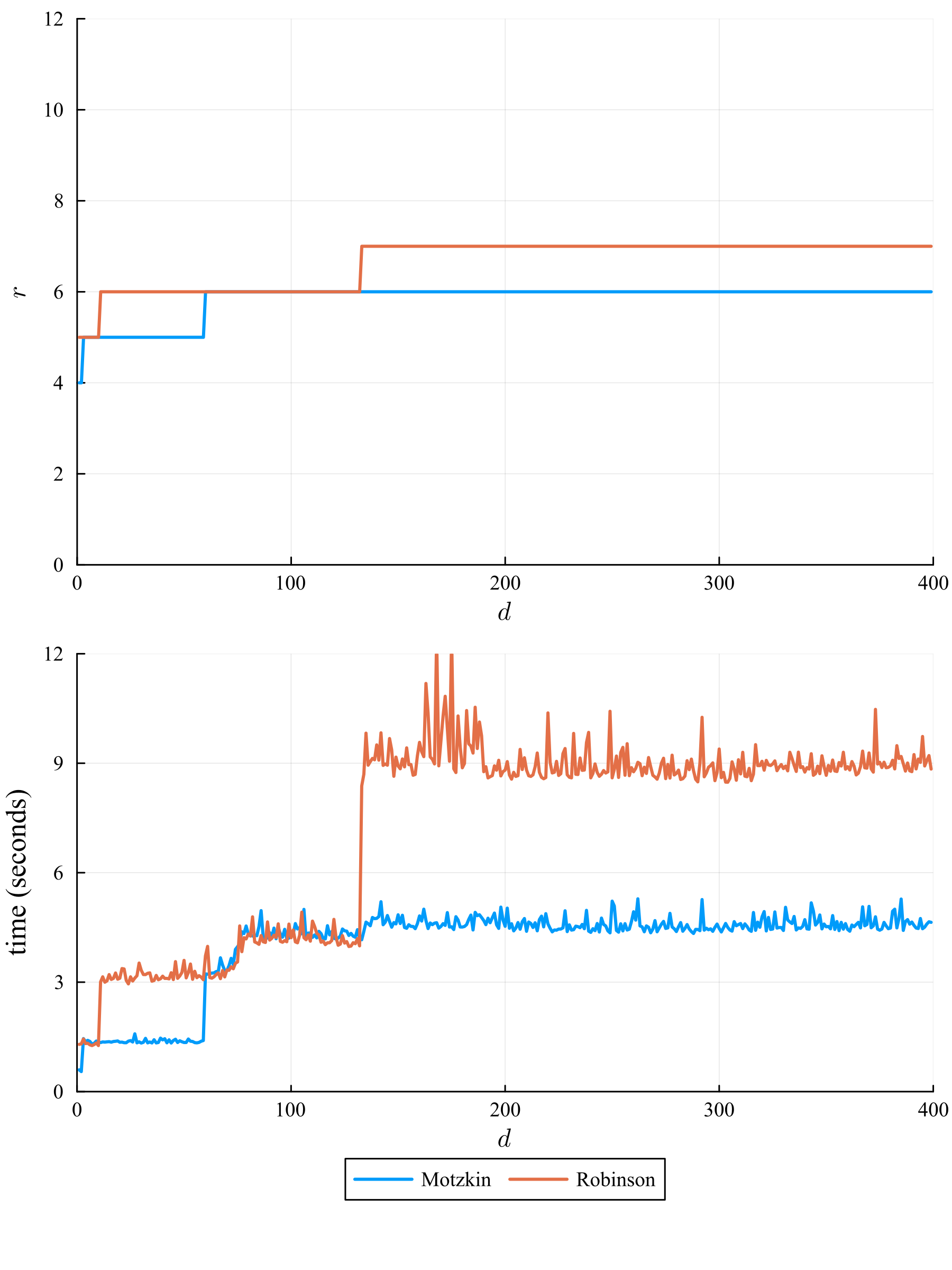}\vspace{-8mm}
  \caption{Performances of implementation of Algorithm~\ref{alg:Lasserre} for lifted Motkzin and Robinson polynomials by $\frac{1}{d}$.} 
  \label{fig:lasserre}
\end{figure}

Figure~\ref{fig:lasserre} reports the parameter $r$ and the time (in seconds) used by the implementation. The parameter $r$ is the minimal degree for $f_{\epsilon}$ in \eqref{eq:globalCeri} to become a sum of squares. The reader should note that as $d$ (horizontal axis) increases, the values of $r$ and time also increase.

\section{Future work}

We conclude this paper by listing several potential research problems to investigate: (1) degree estimation for the certificates constructed by our method; a consequence of this line of research could be new interesting results for effective Putinar's Positivstellensatz; (2) extend the proposed method to deal with polynomials which are non-negative over $\semialgebraic(G)$; (3) complexity analysis of Algorithm~\ref{alg:overall} possibly leading to optimization of our implementation; (4) identify sub cases and conditions to obtain lower degree certificates.

\section{Acknowlegements}

The research of Weifeng Shang and Chenqi Mou was supported in part by National Natural Science Foundation of China (NSFC 12471477). Jose Abel Castellanos Joo and Deepak Kapur were supported in part by U.S. National Science Foundation awards (CCF-1908804) and (CCF-2513374).

\bibliographystyle{elsarticle-harv} 

\bibliography{references.bib}

\end{document}